\documentclass[twoside,a4paper,preprint]{imsart}

\RequirePackage{amsthm,amsmath,amsfonts,amssymb}
\RequirePackage[numbers,sort&compress]{natbib}
\RequirePackage[colorlinks,citecolor=blue,urlcolor=blue]{hyperref}
\RequirePackage{graphicx}

\usepackage{tikz}
\usepackage{tikz-cd}
\usepackage{pb-diagram}
 \usetikzlibrary{positioning,calc,arrows,shapes,automata,backgrounds,petri,trees,through,calc} 
\usepackage{pgfplots}
\usetikzlibrary{decorations.markings}

\startlocaldefs
\theoremstyle{plain}
\newtheorem{theorem}{Theorem}
\newtheorem{lemma}[theorem]{Lemma}

\newtheorem{proposition}[theorem]{Proposition}

\newtheorem*{remark}{Remark}

\newcommand{\crit}{{\rm Crt}}
\newcommand\E{{\mathbb E}}
\newcommand\X{X(\cdot)}
\newcommand{\var}{{\rm Var}}

\newcommand\R{{\mathbb R}}
\newcommand\C{{\mathbb C}}
\newcommand{\indicator}{\mathbf{1}}

\endlocaldefs

\begin{document}

\begin{frontmatter}

\title{Landscape $k$-complexity of isotropic centered Gaussian fields}

\runtitle{Landscape $k$-complexity of isotropic centered Gaussian fields}

\begin{aug}
%%%%%%%%%%%%%%%%%%%%%%%%%%%%%%%%%%%%%%%%%%%%%%%
%% Only one address is permitted per author. %%
%% Only division, organization and e-mail is %%
%% included in the address.                  %%
%% Additional information can be included in %%
%% the Acknowledgments section if necessary. %%
%% ORCID can be inserted by command:         %%
%% \orcid{0000-0000-0000-0000}               %%
%%%%%%%%%%%%%%%%%%%%%%%%%%%%%%%%%%%%%%%%%%%%%%%
\author[A]{\fnms{Jean-Marc}~\snm{Aza\"{\i}s}\ead[label=e1]{ jean-marc.azais@math.univ-toulouse.fr}}
%\author[B]{\fnms{Federico }~\snm{Dalmao}\ead[label=e2]{fdalmao@unorte.edu.uy} %\orcid{0000-0000-0000-0000} }
\and
\author[C]{\fnms{C\'eline }~\snm{Delmas}\ead[label=e3]{celine.delmas.toulouse@inrae.fr}}
%%%%%%%%%%%%%%%%%%%%%%%%%%%%%%%%%%%%%%%%%%%%%%
%% Addresses                                %%
%%%%%%%%%%%%%%%%%%%%%%%%%%%%%%%%%%%%%%%%%%%%%%
\address[A]{IMT, Universit\'e de Toulouse, France.
\printead[presep={,\ }]{e1}}

%\address[B]{DMEL, CENUR Litoral Norte, Universidad de la Rep\'{u}blica, Uruguay
%\printead[presep={,\ }]{e2}}

\address[C]{MIAT, INRAE,  Universit\'e de Toulouse, France.
\printead[presep={,\ }]{e3}}
\end{aug}

%% Abstract
\begin{abstract}
%% Text of abstract
In large dimension, we study the asymptotic behavior of the mean number of critical points with index $k$ below a level $u$ for an isotropic centered Gaussian random field defined on a family of subsets of $\R^d$ depending  on $d$.
We prove the existence of three regimes depending on the speed of growth  of the volume of the parameter set. In the first regime 
the mean number of critical points  decreases exponentially with the dimension.
For the second regime, there exists a critical level $u_c$ such that the mean number of critical points with index $k$ below a level $u$ with $u>u_c$ increases exponentially with the dimension $d$ independently of the index $k$
and decreases exponentially with $d$ when $u<u_c$. In the third regime,
there exists a layered structure depending on the level $u$ considered and on the index $k$ of the critical points.  This behavior is similar  to  the one  encountered  on the sphere by  Auffinger et al. (2013).
In the particular case of the Bargmann-Fock field, only two regimes coexist.
\end{abstract}

%%Graphical abstract
%\begin{graphicalabstract}
%\includegraphics{grabs}
%\end{graphicalabstract}

%%Research highlights
%\begin{highlights}
%\item Research highlight 1
%\item Research highlight 2
%\end{highlights}

%% Keywords
\begin{keyword}
%% keywords here, in the form: keyword \sep keyword
Complexity \sep Critical points \sep Gaussian fields \sep High dimension
%% PACS codes here, in the form: \PACS code \sep code
\kwd{Complexity}
\kwd{Critical points}
\kwd{Gaussian fields}
\kwd{High dimension}
%% MSC codes here, in the form: \MSC code \sep code
%% or \MSC[2008] code \sep code (2000 is the default)

\end{keyword}

\end{frontmatter}

%% Add \usepackage{lineno} before \begin{document} and uncomment 
%% following line to enable line numbers
%% \linenumbers

%% main text
%%

%% Use \section commands to start a section

\section{Introduction}
This note studies the asymptotic behavior  in large dimension of the mean number of critical points with index $k$ below a level $u$. This for  centered Gaussian isotropic  random fields defined on a  family  of subsets $T_d$ of $\R^d$ whose volume $|T_d|$ is logarithmically comparable
to that of a ball as $d$ tends to infinity. More precisely, we assume that
$$
\frac{1}{d}\log \left(\frac{|T_d|}{\kappa_d}\right) \rightarrow V\in \R \mbox{ as }d\rightarrow +\infty,
$$
where $\kappa_d$ denotes the volume of the unit ball in $\R^d$.
$T_d$ may be, for example, the ball of radius $R$ in $\R^d$ or the square $\displaystyle [0,L/\sqrt{d}]^d$ with $0<L,R<+\infty$.
In those cases $\displaystyle R:=\exp(V)$ and $\displaystyle L:=\sqrt{2\pi}\exp(V+1/2)$.
We prove that there exists a critical value $V=V^1_c$, or equivalently a critical radius $R^1_c$ for the ball of radius $R$ or a critical side length $L^1_c$ for the square $[0,L/\sqrt{d}]^d$, such that  the mean number of critical points with index $k$ on $T_d$
decreases exponentially with $d$ when $V<V^1_c$ ($R<R^1_c$ or $L<L^1_c$), is logarithmically comparable to zero when $V=V^1_c$ ($R=R^1_c$ or $L=L^1_c$)
and increases exponentially with $d$ when $V>V^1_c$ ($R>R^1_c$ or $L>L^1_c$).
We obtain
$$
V_c^1=\log \left(\frac{\exp(1) \sqrt{3\lambda_2}}{\sqrt{\lambda_4}}\right)\mbox{, }\hspace{0.5cm} R^1_c=\frac{\exp(1) \sqrt{3\lambda_2}}{\sqrt{\lambda_4}} \hspace{0.5cm} \mbox{ and } \hspace{0.5cm} L^1_c=\exp(3/2)\sqrt{\frac{6\lambda_2 \pi}{\lambda_4}},
$$
where $\lambda_2$ and $\lambda_4$ are respectively the second and fourth spectral moments of the Gaussian field. When $V>V^1_c$, there also exist another critical value $V_c^2>V_c^1$ such that
when $V_c^1<V<V_c^2$ there exist a critical level $v_c$ for which the mean number of critical points with index $k$ below the level $u=v\sqrt{d}$ on $T_d$, $\E(\crit_{u\downarrow}^k(T_d))$,
decreases exponentially with $d$ when $v<v_c$ and increases exponentially with $d$ when $v>v_c$. This behaviour is independent on the index $k$.
When $V>V_c^2$ there exist critical levels $v_c^{0}<v_c^1<v_c^2<\ldots$ for which the mean number
of critical points with index $k$ below the level $u=v\sqrt{d}$ on $T_d$ decreases exponentially with $d$ when $v<v_c^{k}$ and increases exponentially with $d$ when $v>v_c^{k}$.
Thus there exists a layered structure in the distribution of the critical points: there will be approximatively no critical points below $u_0=v_c^0\sqrt{d}$, only critical points with index 0 between $u_0$ and $u_1=v_c^1\sqrt{d}$, only critical points with indexes 0 and 1
between $u_1$ and $u_2=v_c^2\sqrt{d}$ and so on.
We have, as $d\to \infty$,
$$
\E(\crit_{u\downarrow}^k(T_d))\approx \exp(d\times \Theta_k(V, v)),
$$
where $ \approx$ means equivalence of logarithms and $ \Theta_k(V, v)$ is the landscape $k$-complexity of the Gaussian random field on $T_d$. The logarihmic behavior of $\E(\crit_{u\downarrow}^k(T_d))$ as a function of $d$ is related to the sign of $\Theta_k(V, v)$.

Our study is analogous to that of \cite{ABAC} and  \cite{ABA} that considers the landscape $k$-complexity of  centered isotropic Gaussian random fields defined on  the high-dimensional sphere. The particularity of our study is that we consider  Gaussian
random fields defined on ``flat'' subsets $T_d$ of $\R^d$. We prove the existence of two critical values $V_c^1$ and $V_c^2$ for the volume of $T_d$ corresponding to three different regimes for the behavior of $\E(\crit_{u\downarrow}^k(T_d))$.
These critical values do not appear in the study of \cite{ABAC} and \cite{ABA}. Our main tools are the Kac-Rice formula and the Large Deviations Principle for the $k$th smallest eigenvalue of a GOE-matrix, as in \cite{ABAC} and \cite{ABA}.
The classical papers devoted to the ``flat case'', see \cite{F}, \cite{BA}, \cite{AZ},  use a quadratic penalty.
Our approach  is in some sense similar to those approaches.  For example, the quadratic penalty used in \cite{F}  can be interpreted as a constraint that limits the parameters to be in a sphere.  But precise equations  are different, as well as they are different from the ``round case'' of \cite{ABAC} and \cite{ABA}. So, it is worth establishing these equations in this short note.  In addition, papers such as \cite{F} and \cite{BA} do not specify the index.
 
  Our main result is  Theorem  \ref{t:gege:plat:general}. It describes the logarithmic  behavior of the number of critical points as a function of the level. In contrast to \cite{ABAC} and \cite{ABA}, there is no distinction between ``pure-like'' and general   models. The three behaviors can occur depending on the size of the set. The Bargmann-Fock field corresponds to the limit case $V_c^1=V_c^2$. Therefore only two regimes coexist for the behavior of $\E(\crit_{u\downarrow}^k(T_d))$ depending on the volume of $T_d$.

 \section{Assumptions, notations and background}
In this paper $\X = \{X(t) : t \in  \R^d \}$ is a real-valued centered stationary and isotropic Gaussian field with unit variance.
Its covariance function is defined by
 \begin{equation}\label{e:statio:2}
 \E\big( X(s), X(t)\big)=c(s,t)=\mathbf r (\|s-t\| ^2) .%= \rho  (\|s-t\| )
\end{equation}

$\mathbf r$ is assumed to be $\C ^2$. The spectral moments $\lambda_0$, $\lambda_{2}$ and $\lambda_4$ are defined by
\begin{equation}
1= \var(X(t)) = \mathbf r(0),
\lambda_{2n}:=\var\left(\frac{\partial^{n}X(t)}{\partial t_{\ell}^n}\right)= (-1)^n\frac{(2n)!}{n!}\mathbf r^{(n)}(0)  \mbox{ with } \ell=1,\ldots, d \mbox{ and } n=1,2. \label{critical:spectralmoments}
\end{equation}

The covariance function $\mathbf r (\cdot)$ is assumed to  define a valid covariance in any dimension. It will be called    a ``Schoenberg covariance''. Proposition \ref{calc4:p:Schoenberg_plat} below gives a charaterization of such covariances.
 
\begin{proposition} [\cite{schoenberg1938}]\label {calc4:p:Schoenberg_plat}
 A covariance  function defined by \eqref{e:statio:2} is a Schoenberg covariance  with   $\mathbf r(0) =1$   if and only if
  there exists a probability measure  $G$ on $[0, +\infty)$
such that
\begin{equation}\label{e:shon}
\mathbf r(x) = \int_0^{+\infty} e^{-xw} G({\rm d}w)  ~~\mbox{ for all } x
\geq 0.
\end{equation}
\end{proposition}

A Schoenberg covariance satisfies $\lambda_4\geq 3\lambda_2^2$ when these quantities are well defined (see \cite{AW2}). Equality holds if and only if
$\displaystyle \mathbf r(t) =e^{-at}$ for some $a >0$. In that case $\displaystyle \E \left( X(0)X(t)\right)=e^{-a\|t\|^2}$. This random field is named ``Bargmann-Fock field''.
In this paper, we also assume that the paths  of  $\X$  are $\C^2$   and $\lambda_2>0$.
Then by \cite[Proposition 2.2 and Section 7.1]{aal} the sample paths are almost surely Morse and this implies that the number of critical points with
index $k$  in a Borel set $T$,  ${{\crit}^k}(T)$,  is well defined. Note that the recent work of \cite{aal} implies that the conditions that appear in \cite{adlertaylor} or \cite{AW-book} are in fact not necessary
for the sample paths to be almost surely Morse.
We recall that the index is the number of negative eigenvalues of the Hessian denoted here by $i(X''(t))$.
More precisely, we define

\begin{equation}\label{critical:nombre1}
{{\crit}^k}(T):=\#\{t\in T : X'(t)=0,\ i(X''(t))=k \}.
\end{equation}
We define also the number of  critical points of index $k$ below the level $u$ by
\begin{equation}\label{critical:nombre3}
{\crit}^k_{u\downarrow}(T):=\#\{t\in T :  X'(t)=0,\ i(X''(t))=k, \ X(t)<u\}.
\end{equation}

First we recall the following proposition, \cite[Proposition 4]{AD2022}, giving the exact expression of the mean number of critical points with index $k$ below a level $u$ on a set $T$.

\begin{proposition}\label{critical:p:NPC}
Let  $L_p$  be the $p$th ordered eigenvalue of a $(d+1)$-GOE matrix ($L_1\leq L_2\leq \ldots \leq L_{d+1}$). Under assumptions above,
\begin{equation}\label{critical:eqNck}
\E\left({{\crit}^k}(T)\right) =  \frac{\sqrt{2}|T|}{\pi^{(d+1)/2}} \left( \frac{\lambda_4}{3\lambda_2}\right)^{d/2} \Gamma \left( \frac{d+1}{2}\right)\E \left( \exp\left(-\frac{L^2_{k+1}}{2}\right)\right).
\end{equation}
When $\displaystyle \lambda_4-3\lambda_2^2>0$,
\begin{equation}\label{critical:eqNcku}
\E \left ( {{\crit}_{u\downarrow}^k}(T)\right) =\frac{\sqrt{2}|T|}{\pi^{(d+1)/2}} \left( \frac{\lambda_4}{3\lambda_2}\right)^{d/2} \Gamma \left( \frac{d+1}{2}\right)
\times \E \left\{\exp\left(-\frac{L_{k+1}^2}{2}\right){\Phi}\left[\sqrt{\frac{\lambda_4}{\lambda_4-3\lambda_2^2}}\left(u-L_{k+1}\frac{\sqrt{6}\lambda_2}{\sqrt{\lambda_4}}\right)\right]\right\},
\end{equation}
and when $\displaystyle \lambda_4-3\lambda_2^2=0$
\begin{equation}\label{critical:eqNckup}
\E \left({{\crit}_{u\downarrow}^k}(T)\right)=\frac{\sqrt{2}|T|\lambda_2^{d/2}}{\pi^{(d+1)/2}}   \Gamma \left( \frac{d+1}{2}\right)    \E \left( \exp\left(-\frac{L_{k+1}^2}{2}\right) \indicator_{L_{k+1}\leq u/\sqrt{2}} \right).
\end{equation}
\end{proposition}

\begin{remark}
This kind of result was first obtained in the spherical case by \cite{ABA} and \cite{ABAC}.
This result can also be found  in  \cite[Proposition 3.9]{cheng2} written in a slightly different way.
\end{remark}

\section{Logarithmic asymptotics}\label{critical:s:logarithmic_asymptotics_planar}

In this section, we consider the logarithmic asymptotic behavior of $\displaystyle \E \left({\crit}^k_{u\downarrow}(T_d)\right)$ as $d\to \infty$, where $T_d$ is a family of  Borel sets. Roughly speaking, we obtain the same kind of results as in \cite{ABA} and \cite{ABAC}  when the volume of the sets  $T_d$  is logarithmically comparable to that of a ball.  More precisely we assume:
\begin{equation}\label{critical:condition_S}
\frac{1}{d}\log \left(\frac{|T_d|}{\kappa_d}\right) \rightarrow  V \in \R \mbox{ when } d\rightarrow +\infty,
\end{equation}
where $\displaystyle \kappa_d=\frac{\pi^{d/2}}{\Gamma \left(\frac{d}{2}+1\right)}$ denotes the  volume  of the unit ball in $\R^{d}$. $T_d$ may be, for example, the ball of radius $R$ in $\R^d$ or the square $[0,L/\sqrt{d}]^d$ with $0<L<+\infty$.
In these cases $R$ and $L$ are given by  
\begin{equation}\label{def:RetL}
R:=\exp (V)\hspace{0.5cm}\mbox{ and }\hspace{0.5cm}L:=\exp(V+1/2) \sqrt{2\pi}
\end{equation}

We set $\displaystyle \tilde{L}_{k+1}:=\frac{L_{k+1}}{\sqrt{d+1}}$ for $k=0,\ldots d$.
The Large Deviations Principle (LDP) for the $(k+1)$ smallest eigenvalue of a $d$-$GOE$ matrix obtained in \cite[Theorem A.9]{ABAC}, formulated below as a lemma, is the main tool for the results of this section.

\begin{lemma}\label{critical:LDP}
Let $L_1\leq L_2\leq \ldots \leq L_{d+1}$ be the ordered eigenvalues of a $(d+1)$-$GOE$ matrix. We set $\displaystyle \tilde{L}_{k+1}:=\frac{L_{k+1}}{\sqrt{d+1}}$, then
$\displaystyle \tilde{L}_{k+1}$ satisfies the LDP with speed $(d+1)$ and good rate function
\begin{equation}\label{critical:def:J_k}
J_k(x)=(k+1)I(-x),
\end{equation}
with
\begin{equation}\label{def:I1}
I(x):=\left\{\begin{array}{lr}
\displaystyle \int_{\sqrt{2}}^x \sqrt{z^2-2}~dz, & if x\geq \sqrt{2},\\
\displaystyle +\infty,& otherwise.\\
\end{array}
\right.
\end{equation}
\end{lemma}

\subsection{Main result}
We assume $\displaystyle \lambda_4-3\lambda_2^2 > 0$ with the condition~\eqref{critical:condition_S} on the set $T_d$.
We set
\begin{equation}
v:=\frac{u}{\sqrt{d+1}}, \hspace{1cm} E_{\star}:=-\left(\frac{\lambda_4+3\lambda_2^2}{\lambda_2\sqrt{3\lambda_4}}\right), \hspace{1cm}
E_{\star \star}:=-\frac{2\sqrt{3}\lambda_2}{\sqrt{\lambda_4}}. \label{def:E}
\end{equation}
When $E_{\star}^2-(k+1)^2(E_{\star \star}-E_{\star})^2 \neq 0$, we define
\begin{equation}\label{def:x}
x_{k,v}:=\frac{-\sqrt{2}vE_{\star}+\sqrt{2}(k+1)(E_{\star \star}-E_{\star})\sqrt{(k+1)^2(E_{\star \star}-E_{\star})^2+(v^2-E_{\star}^2)}}{E_{\star}^2-(k+1)^2(E_{\star \star}-E_{\star})^2},
\end{equation}
and when $E_{\star}^2-(k+1)^2(E_{\star \star}-E_{\star})^2 = 0$,
\begin{equation}\label{def:xbix}
x_{k,v}:=-\frac{E_{\star}^2+v^2}{\sqrt{2}vE_{\star}}.
\end{equation}
\begin{equation}\label{def:thetastar}
\displaystyle \Theta^{\star}_k(V,v):=
\left \{
\begin{array}{l}
\displaystyle V+\frac{1}{2}\log \left(\frac{\lambda_4}{3\lambda_2}\right)-1 \mbox{ }\mbox{ if }v\geq E_{\star \star}\\
\displaystyle V+\frac{1}{2}\log \left(\frac{\lambda_4}{3\lambda_2}\right)-\left(1-\frac{(v-E_{\star \star})^2}{E_{\star \star}(E_{\star \star}-E_{\star})}\right) \mbox{ } \mbox{ if } E_{\star}\leq v\leq E_{\star \star}\\
\displaystyle V+\frac{1}{2}\log \left(\frac{\lambda_4}{3\lambda_2}\right)-\left(\frac{x_{k,v}^2}{2}-\frac{\left(v+x_{k,v}\frac{E_{\star \star}}{\sqrt{2}}\right)^2}{E_{\star \star}(E_{\star \star}-E_{\star})}\right)-J_k(x_{k,v}) \mbox{ }\mbox{ if } v\leq E_{\star }
\end{array}
\right.
\end{equation}
%\]
where the function $J_k$ is defined in~\eqref{critical:def:J_k}.

\begin{theorem} \label{t:gege:plat:general}
Under assumptions above, we have
\begin{eqnarray*}
\lim_{d\rightarrow +\infty}\frac{1}{d}\log \E\left(\crit^k(T_d)\right)&=&\Theta^{\star}(V):=V+\frac{1}{2}\log \left(\frac{\lambda_4}{3\lambda_2}\right)-1 \mbox{ }\mbox{ and}\\
\lim_{d\rightarrow +\infty}\frac{1}{d}\log \E\left(\crit^k_{u\downarrow}(T_d)\right)&=&\Theta^{\star}_k(V,v).
\end{eqnarray*}
\end{theorem}

\begin{proof}[Proof of Theorem \ref{t:gege:plat:general}]
By Proposition \ref{critical:p:NPC}, setting $\displaystyle g(x):=-x^2/2$, we have
\begin{multline}\label{firstequation}
\lim_{d\rightarrow +\infty}\frac{1}{(d+1)}\log \E\left(\crit_{u \downarrow }^k(T_d)\right)=\\
V+\frac{1}{2}\log \left(\frac{\lambda_4}{3\lambda_2}\right)
+\lim_{d\rightarrow +\infty}\frac{1}{(d+1)}\log \E\left(\Phi \left(\sqrt{\frac{(d+1)\lambda_4}{\lambda_4-3\lambda_2^2}}
\left(v-\tilde{L}_{k+1}\frac{\sqrt{6}\lambda_2}{\sqrt{\lambda_4}}\right) \right)e^{(d+1)g\left(\tilde{L}_{k+1}\right)}\right).
\end{multline}
%where
%$$
%g_d(x):=-\frac{x^2}{2}
%+\frac{1}{(d+1)}\log \Phi \left(\sqrt{d+1}\sqrt{\frac{\lambda_4}{\lambda_4-3\lambda_2^2}}
%\left(v-x\frac{\sqrt{6}\lambda_2}{\sqrt{\lambda_4}}\right) \right)  .
%$$
Moreover
$$
{E}_{\star \star}-{E}_{\star}=\frac{(\lambda_4-3\lambda_2^2)}{\lambda_2\sqrt{3\lambda_4}}.
$$
Then, setting $\displaystyle \kappa:=2/({E}_{\star \star}({E}_{\star}-{E}_{\star \star}))$,
$$
\Phi \left(\sqrt{\frac{(d+1)\lambda_4}{\lambda_4-3\lambda_2^2}}
\left(v-\tilde{L}_{k+1}\frac{\sqrt{6}\lambda_2}{\sqrt{\lambda_4}}\right) \right)e^{(d+1)g\left(\tilde{L}_{k+1}\right)}=
\Phi \left(\sqrt{(d+1)\kappa}\left(v+\tilde{L}_{k+1}\frac{E_{\star \star}}{\sqrt{2}}\right) \right)e^{(d+1)g\left(\tilde{L}_{k+1}\right)}.
$$
First case: $\displaystyle v\geq {E}_{\star \star}$. a) We have
\begin{equation}\label{firstppequation}
\frac{1}{d+1}\log \E\left(\Phi \left(\sqrt{(d+1)\kappa}\left(v+\tilde{L}_{k+1}\frac{E_{\star \star}}{\sqrt{2}}\right) \right)e^{(d+1)g\left(\tilde{L}_{k+1}\right)}\right)\leq \frac{1}{d+1}\log \E\left(e^{(d+1)g\left(\tilde{L}_{k+1}\right)}\right).
\end{equation}
Now we use Varadhan's Integral Lemma \cite[Theorem 4.3.1]{dembo}. Since, by Lemma \ref{critical:LDP}, $\tilde{L}_{k+1}$ satisfies the LDP with good rate funtion $J_k$ and since $g(x)\leq 0$, condition (4.3.3) in \cite{dembo} holds and
\begin{equation}\label{firstpequation}
\lim_{d\rightarrow +\infty} \frac{1}{d+1}\log \E\left(e^{(d+1)g\left(\tilde{L}_{k+1}\right)}\right)=\sup_{x\in \R} (g(x)-J_k(x))=-1.
\end{equation} 
b) We have
$$
\E\left(\Phi \left(\sqrt{(d+1)\kappa}\left(v+\tilde{L}_{k+1}\frac{E_{\star \star}}{\sqrt{2}}\right) \right)e^{(d+1)g\left(\tilde{L}_{k+1}\right)}\right)\geq \E\left(0.5e^{(d+1)g\left(\tilde{L}_{k+1}\right)} \indicator_{\tilde{L}_{k+1}\leq -\frac{\sqrt{2} v}{E_{\star \star}}}\right),
$$
then
\begin{multline}\label{secondequation}
\lim_{d\rightarrow +\infty} \frac{1}{d+1}\log \E\left(\Phi \left(\sqrt{(d+1)\kappa}\left(v+\tilde{L}_{k+1}\frac{E_{\star \star}}{\sqrt{2}}\right) \right)e^{(d+1)g\left(\tilde{L}_{k+1}\right)}\right) \geq \\
\lim_{d\rightarrow +\infty} \frac{1}{d+1}\log \E\left(e^{(d+1)g\left(\tilde{L}_{k+1}\right)} \indicator_{\tilde{L}_{k+1}\leq -\frac{\sqrt{2} v}{E_{\star \star}}}\right).
\end{multline}
Since condition (4.3.3) in \cite{dembo} holds, by \cite[Lemma 4.3.8]{dembo} the tail condition (4.3.2) in \cite{dembo} also holds and by  \cite[exercise 4.3.11]{dembo}
\begin{equation}\label{thirdequation}
\lim_{d\rightarrow +\infty} \frac{1}{d+1}\log \E\left(e^{(d+1)g\left(\tilde{L}_{k+1}\right)} \indicator_{x\leq -\frac{\sqrt{2} v}{E_{\star \star}}}\right)=\sup_{x\leq -\frac{\sqrt{2} v}{E_{\star \star}}} (g(x)-J_k(x))=-1.
\end{equation}
By \eqref{firstequation}, \eqref{firstppequation}, \eqref{firstpequation}, \eqref{secondequation}, \eqref{thirdequation}, we conclude that
$$
\lim_{d\rightarrow +\infty}\frac{1}{(d+1)}\log \E\left(\crit_{u \downarrow }^k(T_d)\right)=
V+\frac{1}{2}\log \left(\frac{\lambda_4}{3\lambda_2}\right)-1.
$$
This gives the first line of ${\Theta}^{\star}_k(V,v)$ in \eqref{def:thetastar} and the expression of $\Theta^{\star}(V)$ in Theorem \ref{t:gege:plat:general}.\\
\\
\noindent Second case $\displaystyle v< {E}_{\star \star}$. a) For $\displaystyle -\frac{v\sqrt{2}}{E_{\star \star}}\leq x \leq -\sqrt{2}$, there exists a constant $C:=\sqrt{\pi/2}$ such that
$$
\Phi \left(\sqrt{(d+1)\kappa}\left(v+x\frac{E_{\star \star}}{\sqrt{2}}\right) \right) \leq C\varphi \left(\sqrt{(d+1)\kappa}\left(v+x\frac{E_{\star \star}}{\sqrt{2}}\right) \right),
$$
where $\varphi$ is the pdf of a standard Gaussian variable. We deduce
$$
\E\left(\Phi \left(\sqrt{(d+1)\kappa}\left(v+\tilde{L}_{k+1}\frac{E_{\star \star}}{\sqrt{2}}\right) \right)e^{(d+1)g\left(\tilde{L}_{k+1}\right)}\right)\leq \E\left(e^{(d+1)\tilde{g}\left(\tilde{L}_{k+1}\right)}\right),
$$
where
$$
\tilde{g}(x)=-\frac{x^2}{2}
-\frac{\kappa}{2}\left(v+x\frac{{E}_{\star \star}}{\sqrt{2}}\right)^2
\indicator_{\{ -\frac{v\sqrt{2}}{{E}_{\star \star}}\leq x \leq -\sqrt{2}\}}.
$$
Therefore
\begin{multline}\label{majoration}
\lim_{d\rightarrow +\infty}\frac{1}{(d+1)}\log \E\left(\Phi \left(\sqrt{(d+1)\kappa}\left(v+\tilde{L}_{k+1}\frac{E_{\star \star}}{\sqrt{2}}\right) \right)e^{(d+1)g\left(\tilde{L}_{k+1}\right)}\right)\leq \\
\lim_{d\rightarrow +\infty}\frac{1}{(d+1)}\log \E\left(e^{(d+1)\tilde{g}\left(\tilde{L}_{k+1}\right)}\right).
\end{multline}
Since $\tilde{g}(x)\leq 0$, condition (4.3.3) in \cite{dembo} holds and by \cite[Theorem 4.3.1]{dembo}, the r.h.s. of \eqref{majoration} is equal to 
$$
\sup_{x\in \R}(\tilde{g}(x)-J_k(x)).
$$
b) For $x\leq 0$, we have $\displaystyle \Phi(x)/\varphi(x) \geq f(x)$ where $\displaystyle f(x):=(-1/x +1/x^3)\indicator_{x\leq -2}+(\Phi(-2)/\varphi(-2))\indicator_{-2\leq x\leq 0}$ is an increasing function. Therefore
\begin{multline*}
\E\left(\Phi \left(\sqrt{(d+1)\kappa}\left(v+\tilde{L}_{k+1}\frac{E_{\star \star}}{\sqrt{2}}\right) \right)e^{(d+1)g\left(\tilde{L}_{k+1}\right)}\right)\geq \\
 \E \left(\frac{f\left(\sqrt{(d+1)\kappa}\left(v-E_{\star \star}\right)\right)}{\sqrt{2\pi}}e^{(d+1)\tilde{g}\left(\tilde{L}_{k+1}\right)}\indicator_{ -\frac{v\sqrt{2} }{E_{\star \star}}\leq \tilde{L}_{k+1}\leq -\sqrt{2}}\right).
\end{multline*}
We deduce
\begin{multline}\label{minoration}
\lim_{d\rightarrow +\infty}\frac{1}{(d+1)}\log \E\left(\Phi \left(\sqrt{(d+1)\kappa}\left(v+\tilde{L}_{k+1}\frac{E_{\star \star}}{\sqrt{2}}\right) \right)e^{(d+1)g\left(\tilde{L}_{k+1}\right)}\right)\geq \\
\lim_{d\rightarrow +\infty}\frac{1}{(d+1)}\log \E\left(e^{(d+1)\tilde{g}\left(\tilde{L}_{k+1}\right)}\indicator_{ -\frac{v\sqrt{2} }{E_{\star \star}}\leq \tilde{L}_{k+1}\leq -\sqrt{2}}\right).
\end{multline}
Since condition (4.3.3) in \cite{dembo} holds, by \cite[Lemma 4.3.8]{dembo} the tail condition (4.3.2) in \cite{dembo} also holds and by  \cite[exercise 4.3.11]{dembo}, the r.h.s. of \eqref{minoration} is equal to 
$$
\sup_{-\frac{v\sqrt{2}}{E_{\star \star}}\leq x \leq -\sqrt{2}}(\tilde{g}(x)-J_k(x)).
$$
It is easy to check that
$$
\sup_{x\in \R}(\tilde{g}(x)-J_k(x))=\sup_{-\frac{v\sqrt{2}}{E_{\star \star}}\leq x \leq -\sqrt{2}}(\tilde{g}(x)-J_k(x)).
$$
In the case $E_{\star}< v < E_{\star \star}$, we can check that $\tilde{g}'(x)>0$ and $-J'_k(x)>0$ $\forall x \in [-\frac{v\sqrt{2}}{E_{\star \star}};-\sqrt{2}]$.
Thus the maximum of $\tilde{g}(x)-J_k(x)$ is obtained for $x=-\sqrt{2}$. This gives the second line of ${\Theta}^{\star}_k(V,v)$ in \eqref{def:thetastar}.\\
\\
When $v<{E}_{\star}$ we look for the maximum of $\displaystyle \tilde{g}(x)-J_k(x)$ on $[-\frac{v\sqrt{2}}{{E}_{\star \star}};-\sqrt{2}]$.
We have $\tilde{g}'(x)>0$ and $-J'_k(x)>0$ $\forall x \in [-\frac{v\sqrt{2}}{E_{\star \star}};-\frac{v\sqrt{2}}{E_{\star}}]$.
Moreover $\tilde{g}'(-\sqrt{2})-J'_k(-\sqrt{2})<0$. So the maximum is reached on $[-\frac{v\sqrt{2}}{{E}_{\star}};-\sqrt{2}]$
\begin{equation}\label{eq:maximum}
\tilde{g}'(x)-J'_k(x)=0 \Leftrightarrow (k+1)\sqrt{x^2-2}=-\frac{(x{E}_{\star}+\sqrt{2}v)}{({E}_{\star \star}-{E}_{\star})}.
\end{equation}
This implies that
$$
\left(\frac{{E}_{\star}^2}{({E}_{\star \star}-{E}_{\star})^2}-(k+1)^2\right)x^2+\frac{2\sqrt{2}v{E}_{\star}}{({E}_{\star \star}-{E}_{\star})^2}x+2(k+1)^2+\frac{2v^2}{({E}_{\star \star}-{E}_{\star})^2}=0.
$$
When $E_{\star}^2-(k+1)^2(E_{\star \star}-E_{\star})^2 \neq 0$, this equation admits two roots
$$
x_{\underset{-}{+}}=\frac{-\sqrt{2}v{E}_{\star}\underset{-}{+}\sqrt{2}(k+1)({E}_{\star \star}-{E}_{\star})\sqrt{(k+1)^2({E}_{\star \star}-{E}_{\star})^2+(v^2-{E}_{\star}^2)}}{{E}_{\star}^2-(k+1)^2({E}_{\star \star}-{E}_{\star})^2}.
$$
When ${E}_{\star}^2-(k+1)^2({E}_{\star \star}-{E}_{\star})^2<0$, we have $x_->0$. The only suitable root is $x_+$.
When ${E}_{\star}^2-(k+1)^2({E}_{\star \star}-{E}_{\star})^2>0$, we can check that $\displaystyle x_-<-\sqrt{2}v/E_{\star}$. The only
suitable root is $x_+$. \\
When $E_{\star}^2-(k+1)^2(E_{\star \star}-E_{\star})^2 = 0$, the root is
$$
x:=-\frac{E_{\star}^2+v^2}{\sqrt{2}vE_{\star}}.
$$
That concludes the proof. 
\end{proof}

\begin{remark}
We can define two critical values, $V_c^1$ and $V_c^2$ such that
\begin{equation}\label{eq:volumec1}
\Theta^{\star}(V_c^1)=\Theta^{\star}_k(V_c^1,E_{\star \star})=0 \Leftrightarrow V_c^1:=\log \left(\frac{\exp(1) \sqrt{3\lambda_2}}{\sqrt{\lambda_4}}\right)
\end{equation}
and
\begin{equation}\label{eq:volumec2}
\Theta^{\star}_k(V_c^2,E_{\star})=0 \Leftrightarrow V_c^2:=\log \left(\frac{\exp \left(E_{\star}/E_{\star \star}\right)\sqrt{3\lambda_2}}{\sqrt{\lambda_4}}\right).
\end{equation}
When the volume of $T_d$ is such that $V<V_c^1$ then $\Theta^{\star}(V)<0$; therefore the mean number of critical points with index $k$ on $T_d$ decreases exponentially with $d$.
When $V>V_c^1$, this mean number increases exponentially with $d$.
\end{remark}

\begin{lemma}\label{lemma:function_theta}
We have the following properties.
\begin{itemize}
\item[a.] $E_{\star}<-2< E_{\star \star} <0$.
\item[b.] For $V$ and $k$ fixed, $\Theta^{\star}_k(V,v)$ is a constant for $v$ above the level $E_{\star \star}$, a strictly increasing function of $v$ not depending on $k$ for $E_{\star}\leq v<E_{\star \star}$
and a strictly increasing function of $v$ depending on $k$ below $E_{\star}$.
\item[c.] For $V$ and $v<E_{\star}$ fixed, $\Theta^{\star}_k(V,v)$ is a strictly decreasing function of $k$.
\end{itemize}
\end{lemma}

\begin{proof}[Proof of Lemma \ref{lemma:function_theta}]
Let us prove that $\Theta^{\star}_k(V,\cdot)$ is strictly increasing on $(-\infty,E_{\star})$.
$$
\frac{\partial \Theta_k^{\star}(V,v)}{\partial v}=-x_{k,v}\frac{\partial x_{k,v}}{\partial v}+\frac{(\sqrt{2}v+x_{k,v}E_{\star \star})(\sqrt{2}+E_{\star \star}\partial x_{k,v} / \partial v)}{E_{\star \star}(E_{\star \star} -E_{\star})}-
J'_k(x_{k,v})\frac{\partial x_{k,v}}{\partial v}.
$$
By \eqref{eq:maximum} we have
$$
-J'_k(x_{k,v})=(k+1)\sqrt{x_{k,v}^2-2}=-\frac{(x_{k,v}E_{\star}+\sqrt{2}v)}{(E_{\star \star} -E_{\star})}.
$$
We deduce, after calculations
$$
\frac{\partial \Theta_k^{\star}(V,v)}{\partial v}=\sqrt{2}\frac{(\sqrt{2}v+x_{k,v}E_{\star \star})}{E_{\star \star}(E_{\star \star} -E_{\star})},
$$
Since $\displaystyle x_{k,v}>-\frac{v\sqrt{2}}{E_{\star \star}}$ (see the end of the proof of Theorem \ref{t:gege:plat:general}), we deduce the strict positivity of $\frac{\partial \Theta_k^{\star}(V,v)}{\partial v}$
for $v\leq E_{\star}$. The other elements of items a. and b. are easy to check.\\
Let us prove item c. Functions $J_k$, $x_{k,v}$ and $\Theta_k^{\star}$ as functions of $k$ can be extended on real numbers. So we can derive with respect to $k$. By direct calculations we have
$$
\frac{\partial \Theta^{\star}_k(V,v)}{\partial k}=\frac{\partial x_{k,v}}{\partial k}\left(\frac{x_{k,v}E_{\star}+\sqrt{2}v+(k+1)(E_{\star \star}-E_{\star})\sqrt{x_{k,v}^2-2}}{E_{\star \star}-E_{\star}}\right)-I(-x_{k,v}).
$$
By \eqref{eq:maximum} we obtain
$$
\frac{\partial \Theta^{\star}_k(V,v)}{\partial k}=-I(-x_{k,v})<0.
$$
\end{proof}

\begin{remark}
\begin{itemize}
\item[1.] $\Theta^{\star}_k(V,v)$ is a strictly increasing function of $V$ for $v$ and $k$ fixed. Therefore, because of \eqref{eq:volumec1} and \eqref{eq:volumec2}, for $V_c^1<V<V_c^2$, we have $\Theta^{\star}_k(V,E_{\star \star})>0$
and $\Theta^{\star}_k(V,E_{\star})<0$. Since $\Theta^{\star}_k(V,v)$ is an strictly increasing function of $v$ for $E_{\star}<v<E_{\star \star}$ and $V$ and $k$ fixed, there exists a critical level $v_c$ such that $\Theta^{\star}_k(V,v_c)=0$. It is easy to check that
\begin{equation}\label{eq:vc}
v_c:=E_{\star \star}-\sqrt{\left(V_c^1-V\right)E_{\star \star}(E_{\star \star}-E_{\star})}.
\end{equation}
When $v>v_c$, the mean number of critical points with index $k$ below the level $u=v\sqrt{d+1}$ increases exponentially with $d$. It decreases exponentially with $d$ when $v<v_c$ independently of $k$.
\item[2.] In the same way, we can check that for $V>V_c^2>V_c^1$ we have $\displaystyle \Theta^{\star}_k(V,E_{\star})>0$ and there exist critical values $v_c^0<v_c^1<v_c^2<\ldots<E_{\star}$ such that
$$
\Theta^{\star}_k(V,v_c^k)=0.
$$
We also have
$$
\Theta^{\star}_k(V,v)<0 \mbox{ for }v<v_c^k \mbox{ and }\Theta^{\star}_k(V,v)>0 \mbox{ for }v>v_c^k.
$$
Then the mean number
of critical points with index $k$ below the level $u=v\sqrt{d+1}$ decreases exponentially with $d$ when $v<v_c^k$ and increases exponentially with $d$ when $v>v_c^k$.
Therefore there exists a layered structure such that below the level $v_c^0$, there will be approximatively no critical points, between
$v_c^0$ and $v_c^1$, only critical points with index 0, between $v_c^1$ and $v_c^2$, only critical points with indexes 0 and 1 and so on.
\end{itemize}
\end{remark}

\begin{remark}
When $\displaystyle \lambda_4-3\lambda_2^2 = 0$, we have $\displaystyle  \E(X(0),X(t))) =
{\bf{r}}(\|s-t\|^2)=\exp \left(-a \|s-t \|^2\right), \mbox{ }a>0$. This random field is called the "Bargmann-Fock field". We have  $E_{\star}=E_{\star \star}=-2$
and only two regimes coexist. We define the critical value $\displaystyle V_c:=1-\log \left({\sqrt{\lambda_2}}\right)$. When $V<V_c$ the mean number of critical points with index $k$ decreases exponentially with $d$
whereas when $V>V_c$ it increases exponentially with $d$. In that case there also exist critical values $v_c^0<v_c^1<v_c^2<\ldots <-2$ such that $\displaystyle {\Theta}^{\star}_k(V,v_c^k)=0$
and we retrieve the layered structure as in the general case above.
\end{remark}

Figure \ref{critical:Matern4l1V3} (c) displays the function ${\Theta}^{\star}_k(V,v)$ for the Bargmann-Fock field when $a=5$. In this case $V_c\simeq -0.15$ and we took $V=0$ to illustrate the regime $V>V_c$. 

\subsection{Example: the Mat\'ern Gaussian field}
In this section $X(\cdot)$ is a centered Mat\'ern Gaussian random field with covariance function satisfying
$$
\E (X(s)X(t))=c(s,t)={\bf{r}}(\|s-t\|^2):=\frac{\sigma^2}{2^{\nu-1}\Gamma(\nu)}\left(\frac{\sqrt{2\nu \|s-t\|^2}}{\ell}\right)^{\nu}K_{\nu}\left(\frac{\sqrt{2\nu \|s-t\|^2}}{\ell}\right),
$$
where $K_{\nu}$ is the modified Bessel function of the second kind. This covariance function has three positive parameters $\sigma^2$, $\ell$ and $\nu$ controlling respectively
the variance, the spatial range and the smoothness. We take $\sigma^2:=1$. We assume $\nu>2$.
We have (see Proposition 2.1 in \cite{cheng})
$$
{\bf{r}}'(0)=-\frac{\nu}{2(\nu-1)\ell^2} \hspace{0.5cm} \mbox{ and }\hspace{0.5cm}{\bf{r}}''(0)=\frac{\nu^2}{4(\nu-1)(\nu-2)\ell^4}.
$$
By \eqref{critical:spectralmoments} we deduce
$$
\lambda_2=\frac{\nu}{(\nu-1)\ell^2}\hspace{0.5cm} \mbox{ , }\hspace{0.5cm}\lambda_4=\frac{3\nu^2}{(\nu-1)(\nu-2)\ell^4}\hspace{0.5cm} \mbox{ and }\hspace{0.5cm}\lambda_4-3\lambda_2^2=\frac{3\nu^2}{(\nu-1)^2(\nu-2)\ell^4}>0.
$$
Figure \ref{critical:Matern4l1V3} [cases (a) and (b)] displays the function $\Theta^{\star}_k(V,v)$
when $\nu=4$ and $\ell=1$ for two different values of $V$. In that case $\lambda_2=4/3$, $\lambda_4=8$,
$E_{\star}\simeq-2.04$, $E_{\star \star}\simeq -1.63$
and $V_c^1\simeq 0.65$, $V_c^2\simeq 0.90$. So we took $V=0.75$ (a) and $V=3$ (b) to illustrate the regimes
$V_c^1<V<V_c^2$ and $V>V_c^2$.

\begin{figure}[h!]
\begin{center}
\begin{tikzpicture}[scale=0.5]
\begin{axis}[xmin=-3, xmax=-1.5, xtick={-3,-2.5}, ytick={-2,-1.5,-1,-0.5}, ymin=-2,ymax=0.5,xlabel={$v$},ylabel={$\Theta^{\star}_k(V,v)$},title={(a)},yscale=1,xscale=1, extra y ticks={0}, extra y tick style={grid=major},
enlarge x limits=false, extra x ticks={-2.04,-1.89,-1.63}, extra x tick labels={${E}_{\star}$,$v_c$, ${E}_{\star \star}$}, extra x tick style={grid=major}]
 \addplot[mark=none,draw=black] table[x=v,y=theta0]{Matern4l1V3quart.txt};
 \addplot[mark=none,draw=green] table[x=v,y=theta1]{Matern4l1V3quart.txt};
 \addplot[mark=none,draw=red] table[x=v,y=theta100]{Matern4l1V3quart.txt};
  \end{axis}
\end{tikzpicture}
\hspace{0.5cm}
\begin{tikzpicture}[scale=0.5]
\begin{axis}[xmin=-4, xmax=-1.5, xtick={-4,-3.5,-3,-2.5}, ytick={-3, -2,-1,1,2}, ymin=-3,ymax=2.5,xlabel={$v$},ylabel={$\Theta^{\star}_k(V,v)$},title={(b)},yscale=1,xscale=1, extra y ticks={0}, extra y tick style={grid=major},
enlarge x limits=false, extra x ticks={-2.04,-1.63}, extra x tick labels={${E}_{\star}$, ${E}_{\star \star}$}, extra x tick style={grid=major}]
 \addplot[mark=none,draw=black] table[x=v,y=theta0]{Matern4l1V3.txt};
 \addplot[mark=none,draw=green] table[x=v,y=theta1]{Matern4l1V3.txt};
 \addplot[mark=none,draw=red] table[x=v,y=theta100]{Matern4l1V3.txt};
  \end{axis}
\end{tikzpicture}
\hspace{0.5cm}
\begin{tikzpicture}[scale=0.5]
 \begin{axis}[xmin=-2.2, xmax=-1.95, xtick={-2.2,-1.95}, ytick={-0.1, 0.1}, ymin=-0.15,ymax=0.16,xlabel={$v$},ylabel={${\Theta}^{\star}_k(V,v)$},title={(c)},yscale=1,xscale=1, extra y ticks={0}, extra y tick style={grid=major},
enlarge x limits=false, extra x ticks={-2.12, -2.104, -2}, extra x tick labels={$v_0$, $v_1$, $-2$}, extra x tick style={grid=major}]
 \addplot[mark=none,draw=black] table[x=v,y=theta0]{Logarithmic_planar_fct1.txt};
 \addplot[mark=none,draw=green] table[x=v,y=theta1]{Logarithmic_planar_fct1.txt};
 \addplot[mark=none,draw=red] table[x=v,y=theta100]{Logarithmic_planar_fct1.txt};
  \end{axis}
\end{tikzpicture}
\caption[Function $\Theta^{\star}_k(V,v)$]{Cases (a) and (b): Function $\Theta^{\star}_k(V,v)$ for $V=3/4$ (a) and $V=3$ (b), ${\bf{r}}(\|s-t\|^2)$
 being the Mat\'ern covariance function with $\nu=4$ and $\ell=1$, $k=0$ in black, $k=1$ in green and $k=100$ in red. Case (c): Function $\Theta^{\star}_k(V,v)$  for ${\bf{r}}(\|s-t\|^2)=\exp \left(-5\|s-t\|^2\right)$, $V=0$, $k=0$ in black, $k=1$ in green and $k=100$ in red.  } \label{critical:Matern4l1V3}
\end{center}
\end{figure}

\end{document}